\documentclass{article}

\usepackage{arxiv}

\usepackage[utf8]{inputenc} % allow utf-8 input
\usepackage[T1]{fontenc}    % use 8-bit T1 fonts
\usepackage{hyperref}       % hyperlinks
\usepackage{url}            % simple URL typesetting
\usepackage{booktabs}       % professional-quality tables
\usepackage{amsfonts}       % blackboard math symbols
\usepackage{nicefrac}       % compact symbols for 1/2, etc.
\usepackage{microtype}      % microtypography
\usepackage{lipsum}
\usepackage{graphicx}
\graphicspath{ {./images/} }

\usepackage{amssymb,amsmath,amsthm}
\usepackage{comment}
\usepackage{float}
\usepackage{xcolor}
\usepackage{ulem}
\usepackage{subcaption}
\usepackage{orcidlink}

\newcommand{\real}{{\rm I\!R}}
\newtheorem{theorem}{Theorem}[section]
\newtheorem{proposition}[theorem]{Proposition}
\newtheorem{corollary}[theorem]{Corollary}
\newtheorem{lemma}[theorem]{Lemma}

\title{Bounds on $A_\alpha-$eigenvalues using graph invariants}

\author{
    João Domingos Gomes da Silva Junior\orcidlink{https://orcid.org/0000-0002-1745-0302}\\
    Mathematics Department\\
    Colégio Pedro II\\
    Campo de São Cristóvão 177, Rio de Janeiro, Brazil \\
  \texttt{joao.dgomes@gmail.com} \\
   \And
    Carla Silva Oliveira\orcidlink{https://orcid.org/0000-0001-6684-8811} \\
    Mathematics Department\\
    ENCE/IBGE\\
    Rua André Cavalcanti 106, Rio de Janeiro, Brazil \\
  \texttt{carla.oliveira@ibge.gov.br} \\
  \And
    Liliana Manuela Gaspar Cerveira da Costa\orcidlink{https://orcid.org/0000-0002-5258-1447} \\
    Mathematics Department\\
    Colégio Pedro II\\
    Campo de São Cristóvão 177, Rio de Janeiro, Brazil \\
  \texttt{lmgccosta@gmail.com} \\
}

\begin{document}
\maketitle
\begin{abstract}
In 2017, Nikiforov introduced the concept of the $A_{\alpha}-$matrix, as a linear convex combination of the adjacency matrix and the degree diagonal matrix of a graph. This matrix has attracted increasing attention in recent years, as it serves as a unifying structure that combines the adjacency matrix and the signless Laplacian matrix. In this paper, we present some bounds for the largest and smallest eigenvalue of $A_{\alpha}-$matrix involving invariants associated to graphs.
\end{abstract}

% keywords can be removed

\section{Introduction}
\label{sec:introduction}

Spectral Graph Theory (SGT) is an essential field in discrete mathematics that studies the properties of graphs, focusing on the eigenvalues and eigenvectors of the matrices associated to graphs, such as the adjacency matrix, the Laplacian matrix and the signless Laplacian matrix among others. These matrices provide powerful tools for analyzing graphs and their applications in various domains.

Nikiforov introduced in \cite{VN17} a family of matrices, $A_\alpha-$matrix, obtained through the convex linear combination the adjacency matrix and the degree matrix of a graph. This matrix is at the state of the art in SGT and many researchers have therefore endeavoured to study it. Among the various problems studied involving the $A_\alpha-$matrix, we can point out: obtaining the smallest value of $\alpha$ for which $A_\alpha$ is positive semidefinite (\cite{BRONDANI2020, VN17, NIKIFOROV2017156}), obtaining the $A_\alpha-$spectrum for certain families of graphs (\cite{ BRONDANI2019209, VN17,  MUHAMMAD2020}), determining the $A_\alpha-$spectrum for the graph resulting from operations between other graphs (\cite{CHEN2019343, Li2019TheS, LIN2018210, VN17}), searching for graphs determined by their $A_\alpha-$spectrum (\cite{CHEN2019343, Lin2017GraphsDB, LIU2018274, Tahir2018}), obtaining lower and upper bounds for the eigenvalues of the $A_\alpha-$matrix (\cite{daSilvaJúniorOliveiradaCosta+2024, LIN2018430, LIN2018210, LIU2020111917, LIU2020347, SP2022, joao2023, WANG2020210}) and, in addition to these, obtaining bounds for the $A_\alpha-$eigenvalues of line graphs (\cite{daSilvaJúniorOliveiradaCosta+2024}).

As we can see, combinations of matrices are therefore crucial in SGT, boosting understanding of the properties of graphs and accelerating algorithmic advances in various applications. They allow researchers to explore the complex structures and dynamics that underlie networks, promoting a deeper understanding of graphs and their properties. Thus, the importance of matrix combinations in SGT remains evident, providing essential tools to unlock the full potential of analyzing and applying graphs.

This paper introduces bounds for the largest and smallest eigenvalues of the $A_\alpha-$matrix that use types of invariants of graphs. The structure of the paper is as follows: in Section~\ref{sec::prelim}, some basic definitions and fundamental results on Linear Algebra, Graph Theory, SGT and about the $A_\alpha-$matrix are provided. In Section~\ref{sec::newbounds}, some new bounds on $A_\alpha-$eigenvalues are introduced.

\section{Preliminaries}\label{sec::prelim}

In this section, we present the fundamental definitions and results of Linear Algebra, Graph Theory and SGT for a better comprehension of the article.

Let $x \in \real^n$, we denote by $\vert x \vert$ the Euclidean norm of $x$. Let $M$ be a $n \times n$ matrix. If $M$ is symmetric, the $M-$eigenvalues are real and we shall index them in non-increasing order, represented by $\lambda_1(M) \geq \lambda_2(M) \geq \dots \geq \lambda_n(M)$. The collection of $M-$eigenvalues together with their multiplicities is called the $M-$spectrum, denoted by $\sigma(M)$. Let \( P \) be a set, and let \( \pi = \{P_1, \ldots, P_m\} \) be a partition of \( P \). Suppose the matrix \( M \) has its rows and columns indexed by the elements of \( P \). The partition of \( M \) induced by \( \pi \) is called an equitable partition if, for each submatrix \( M_{ij} \) formed by the rows indexed by \( P_i \) and the columns indexed by \( P_j \), the row sums of \( M_{ij} \) are constant and equal to \( q_{ij} \). Additionally, the matrix \( N = (q_{ij})_{1 \leq i,j \leq m} \) is referred to as the quotient matrix of \( M \) with respect to \( \pi \). Theorem \ref{EquitPart} relates the $M-$eigenvalues with the $N-$eigenvalues and Theorem~\ref{theo::equitable_eigenvalues_relation} introduces a relation between its largest eigenvalue.

\begin{theorem} \label{EquitPart}
    \cite{ADB_Equitable} Let $M$ be a square matrix of order $n$ and suppose that $M$ has an equitable partition $\pi = \{P_1,P_2,\ldots,P_k\}$. Let $N$ be the quotient matrix of $M$ with respect to the partition $\pi$. Then the eigenvalues of $N$ are eigenvalues of $M$.
\end{theorem}

\begin{theorem} \label{theo::equitable_eigenvalues_relation}
    \cite{YOU201921} Let \( N \) be the equitable quotient matrix of \( M \) with respect to a partition. If \( N \) is a nonnegative matrix and \( M \) is irreducible. Then
    \[
        \lambda_1(M) = \lambda_1(N).
    \]
\end{theorem}

A finite graph is a pair of objects consisting of a nonempty finite set $V=V(G)$ and a subset (possibly empty) $E = E(G)$ of unordered pairs of elements of $V$. Usually written as $G = (V, E)$, the elements of $V$ are the vertices of $G$ and the elements of $E$ are its edges. If $u$ and $v$ denote vertices of $G$, most of the time we use $uv$ instead of $\{u, v\}$ to denote the edge incident to $u$ and $v$. The cardinalities of the sets $V$ and $E$ are denoted, in the proper order, by $n = |V|$ and $m = |E|$, where the former indicates the order of the graph and the latter its size. A graph $G$ is simple when it has at most one edge between any two vertices and no edge starts and ends at the same vertex. In other words, a simple graph is a graph without loops and multiple edges.

If $v_iv_j \in E$, we say that $v_i$ is adjacent to (or neighbor of)  $v_j$ (or $v_j$ is adjacent to $v_i$) and denote it by $v_i \sim v_j$, otherwise, $v_i \nsim v_j$. The set of vertices adjacent to $v_i$ is denoted by $N(v_i)$. The degree of $v \in V$, denoted by $d(v)$, is the number of neighbors of $v$. The minimum degree of $G$ is $\delta(G) = \min \{ d(v)| v \in V \}$ and the maximum degree of $G$ is $\Delta(G) = \max\{d(v)|  v \in V\}$. For short, we use $d_i$, $\Delta$, and $\delta$ to represent $d(v_i)$, $\Delta(G)$, and $\delta(G)$, respectively. Without loss of generality, we assume that the vertices are labeled so that $\Delta = d_1 \geq d_2 \geq \dots \geq d_n = \delta$ and the degree sequence is defined by $d(G) = (d_1, \ldots, d_n)$. The \textit{mean degree} of \( G \) is defined by \( \bar{d} = \dfrac{2m}{n}.\) The second maximum degree of $G$ is defined by $\Delta_2 = d_i \leq \Delta,$ for some $1 \leq i \leq n.$ A graph \( G \) is said to be \( r \)-regular if every vertex in \( G \) has degree \( r \). If \( r = n - 1 \), the graph, denoted by \( K_n \), is called \textit{complete graph}. A graph \( G \) is called semiregular if and only if its degree sequence contains exactly two distinct values, and irregular if and only if its degree sequence contains at least two distinct values. Below, in Figure~\ref{fig:main}, we present examples of graphs illustrating the concepts related to the parameters $\Delta$, $d_2$ and $\Delta_2$.

\begin{figure}[H]
    \centering
    % Subfigure 1
    \begin{subfigure}{0.3\textwidth}
        \centering
        \includegraphics[width=0.4\linewidth]{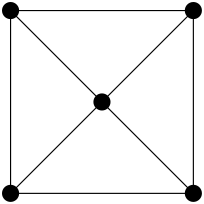}
        \caption{}
        \label{fig:sub1}
    \end{subfigure}
    %\hspace{0.05\textwidth} % Horizontal space between subfigures
    % Subfigure 2
    \begin{subfigure}{0.3\textwidth}
        \centering
        \includegraphics[width=0.4\linewidth]{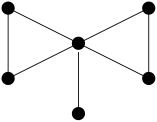}
        \caption{}
        \label{fig:sub2}
    \end{subfigure}
    %\hspace{0.05\textwidth}
    % Subfigure 3
    \begin{subfigure}{0.3\textwidth}
        \centering
        \includegraphics[width=0.6\linewidth]{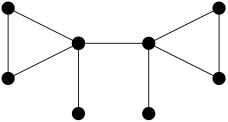}
        \caption{}
        \label{fig:sub3}
    \end{subfigure}
    \caption{Examples of graphs to illustrate the concepts of $\Delta$, $d_2$ and $\Delta_2$}
    \label{fig:main}
\end{figure}

In the graph shown in Figure~\ref{fig:sub1} is semiregular and its degree sequence is $(4, 3, 3, 3, 3)$, where $\Delta = 4$ and $d_2 = \Delta_2 = \delta = 3$. In Figure~\ref{fig:sub2}, the graph is irregular and its degree sequence is $(5, 2, 2, 2, 2, 1)$, with $\Delta = 5$ and $d_2 = \Delta_2 = 2$. Finally, in Figure~\ref{fig:sub3}, the graph is either irregular with degree sequence equal to $(4, 4, 2, 2, 2, 2, 1, 1)$, where $\Delta = d_2 = 4$ and $\Delta_2 = 2$. It is worth noting that if $\Delta_2 = \Delta,\ G$ is a $\Delta-$regular graph.

A graph $G$ is \textit{bipartite} if its vertex set $V$ can be partitioned into two nonempty subsets $V_1$ and $V_2$  (i.e., $V_1\cup V_2=V$ and $V_1\cap V_2=\varnothing$) such that each edge of $G$ has one endpoint in $V_1$ and one endpoint in $V_2$. The \textit{complete bipartite graph} of order $n = n_1 + n_2$ is a graph whose vertices in each partition are neighbors of all the vertices in the other partition. This graph is denoted by $K_{n_1,n_2}$ and, in particular, the \textit{star graph} is denoted by $K_{1,n-1}$. Furthermore, if all vertices in \( V_1 \) have degree \( r \), and all vertices in \( V_2 \) have degree \( s \), \( G \) is called a \((r, s)-\)semiregular bipartite graph, and if \( r = s \), \( G \) reduces to an \( r-\)regular bipartite graph.

A \textit{clique}, $C$, in $G$ is a subset of $V(G)$ such that every two distinct vertices are adjacent. This is equivalent to the condition that the induced subgraph of $G$ induced by $C$ is a complete graph. A maximum clique of $G$ is a clique, such that there is no clique with more vertices. Moreover, the \textit{clique number} of $G$, $\omega(G) = \omega$, is the number of vertices in a maximum clique in $G$. A \textit{walk} of length $k$ in $G$, from $v_i$ to $v_j$, is a sequence of $k+1$ of its vertices $v_i = u_0, u_1, u_2, \ldots, u_k = v_j$, such that $u_{i-1}u_i \in E$, $1 \leq i \leq k$. A path is a walk in which all vertices are distinct. We denote the \textit{path} with $n$ vertices by $P_n$. A graph $G$ is connected if $G$ is an isolated vertex or for every pair of its vertices, there is a path connecting them.
Let $u$ and $w$ be vertices of a connected graph $G$, the distance between $u$ and $w$, $d(u, w)$, is the length of the shortest path in $G$ from $u$ to $w$. The diameter of $G$, denoted by $\text{diam}(G)$, is the greatest distance between two vertices of $G$, i.e., $\text{diam}(G) = \displaystyle \max_{u,w \in V(G)}d(u,w)$. 

The adjacency matrix of $G$,  $A = A(G) = [a_{ij}]$, is a square and symmetric matrix of order $n$, such that $a_{ij} = 1$ if $v_i \sim v_j$ and $a_{ij} = 0$ otherwise. Their eigenvalues are denoted by $\lambda_i(A(G))$, for $1 \leq i \leq n$. The degree matrix of $G$, denoted by $D(G)$, is the diagonal matrix that has the degree of vertex $v_i$, $d(v_i)$, in the $i$-th position. The signless Laplacian matrix is defined by $Q(G) = D(G) + A(G)$ whose eigenvalues are denoted by $\lambda_i(Q(G))$, for $1 \leq i \leq n$. An interesting problem in SGT is to obtain bounds for $A$-eigenvalues and $Q$-eigenvalues involving invariants associated with them. Theorem~\ref{theo::bounds_adjacency} introduces bounds for $\lambda_1(A(G)).$

\begin{theorem}\label{theo::bounds_adjacency}
(\cite{Cvetkovic1990}) Let \( \delta, \bar{d}, \Delta \) be the minimal, mean, and maximal values of the vertex degrees in a connected graph \( G,\) respectively. If \( \lambda_1(A(G)) \) is the largest eigenvalue of the adjacency matrix of \( G \), then
\[
\delta \leq \bar{d} \leq \lambda_1(A(G)) \leq \Delta.
\]
Equality in one place implies equality throughout; and this occurs if and only if \( G \) is regular.
\end{theorem}

The following are presented some results that will be used in Section \ref{sec::newbounds}. The first one, presented by Motzkin and Straus (\cite{Motzkin_Straus_1965}), in 1965, proved the following theorem:

\begin{theorem}\label{theo::motzkin}
    (\cite{Motzkin_Straus_1965}) Let $G$ be a graph with $n$ vertices and $x \in \real^n$ be a non-negative vector. If the maximum clique of graph $G$ has $\omega$ vertices, then
    \begin{equation}
        \max_{x}x^TAx = 1 - \dfrac{1}{\omega}
    \end{equation}
\end{theorem}

Regarding the vertices degrees, the first and second \textit{Zagreb indices} are defined respectively, in \cite{GUTMAN1972} as follows
\begin{align*}
    Z_1(G) &= \sum_{i=1}^n d^2(v_i) \quad \text{and} \\
    Z_2(G) &= \sum_{uv \in E(G)}d(u)d(v).
\end{align*}
The general first Zagreb index is defined by $\displaystyle Z^{(p)}(G) = \sum_{i=1}^n d^{p}(v_i)$, for $p \in \real$, $p \neq 0$, and $p \neq 1$, and seems to have been considered, firstly, by \cite{Li2004} and \cite{Li2005}. For $p = 2$, we have $Z^{(2)}(G) = Z_1(G)$ and the study of its bounds and properties can be found in \cite{CIOABA20061959, DasZagreb, DAS200457, NKMT2003}. The next two results introduce bounds for $Z_1(G)$. 

\begin{theorem} \label{theo::sharp_DAS}
(\cite{Das2003SHARPBF}) Let $G$ be a simple graph with $n$ vertices, $m$ edges, minimum degree $\delta$, and maximum degree $\Delta$. Then, for $n \geq 3$,
\begin{equation*}
    Z_1(G) \geq \Delta^2 + \delta^2 + \frac{(2m - \Delta - \delta)^2}{n-2}.
\end{equation*}
Moreover, equality occurs if and only if $d_2 = \dots = d_{n-1}$.
\end{theorem}

\begin{theorem} \label{theo::sharp_DAS_upper}
(\cite{Das2003SHARPBF}) Let $G$ be a connected graph with $n$ vertices, $m$ edges, and minimum degree $\delta$. Then,
\begin{equation*}
    Z_1(G) \leq 2mn -n(n-1)\delta + 2m(\delta-1).
\end{equation*}
Moreover, equality occurs if and only if $G$ is a star graph or a regular graph.
\end{theorem}

The $A_\alpha-$matrix is a convex linear combination of $A(G)$ and $D(G)$, i.e. 
$$A_\alpha(G) = \alpha D(G) + (1-\alpha)A(G), \text{ for any } \alpha \in [0,1].$$ Particularly, $A_0(G) = A(G)$, $A_1(G) = D(G)$ and $2A_{\frac{1}{2}}(G) = Q(G)$. Therefore, the family of matrices $A_\alpha(G)$ clarifies a unified theory encompassing the matrices $A(G)$ and $Q(G)$. Consequently, $A(G)$, $Q(G)$ and $D(G)$ take on a new perspective, continuing to incite exploration among researchers. It is worth noting that the $A_\alpha-$ matrix is a subset of the generalized adjacency matrices, presented in \cite{VANDAM2003241}, and the universal adjacency matrix in \cite{HAEMERSUniversal}. The eigenvalues of $A_\alpha-$ matrix of a graph $G$, $A_\alpha-$ eigenvalues, are denoted by $\lambda_k(A_\alpha(G)),$ for $k = 1, \ldots, n,$ such that $\lambda_1(A_\alpha(G)) \geq \lambda_2(A_\alpha(G)) \geq \ldots \geq \lambda_n(A_\alpha(G))$ and $A_\alpha-$ spectrum is the set of $A_\alpha-$ eigenvalues.

Below we present some results involving the $A_\alpha-$matrix that will later be used to obtain some bounds for $A_\alpha-$eigenvalues. As $A_\alpha(G)$ is a real and symmetric matrix, Proposition~\ref{prop:rayleigh_alpha} follows directly from Rayleigh's principle.

\begin{proposition} \label{prop:rayleigh_alpha}
(\cite{VN17}) If $\alpha \in [0,1]$ and $G$ is a graph of order $n$, then
\begin{equation}
\lambda_1(A_\alpha(G)) = \max_{x \neq 0_{\real^n}}\dfrac{ x^TA_\alpha(G)x}{x^Tx} \text{\;\; and \;\;}  \lambda_n(A_\alpha(G)) = \min_{x \neq 0_{\real^n}} \dfrac{ x^TA_\alpha(G)x}{x^Tx}.
\end{equation}
Moreover, if $x$ is a unit vector, then $\lambda_1(A_\alpha(G)) = x^TA_\alpha(G)x$ if, and only if, $x$ is an eigenvector of $\lambda_1(A_\alpha(G))$, and $\lambda_n(A_\alpha(G)) = x^TA_\alpha(G)x$ if, and only if, $x$ is an eigenvector of $\lambda_n(A_\alpha(G))$.
\end{proposition}

The next proposition introduces the Perron-Frobenius results for the $A_{\alpha}-$matrix.

\begin{proposition}  \label{prop::Perron_alpha}
 (\cite{VN17}) Let $\alpha \in [0, 1)$, $G$ be a graph, and $x$ be a nonnegative eigenvector to $\lambda_1(A_{\alpha}(G))$.
\begin{enumerate}
    \item[(a)] If $G$ is connected, then $x$ is positive and is unique up to scaling.
    \item[(b)] If $G$ is not connected and $P$ is the set of vertices with positive entries in $x$, then the subgraph induced by $P$ is a union of components $H$ of $G$ with $\lambda_1(A_{\alpha}(H)) = \lambda_1(A_{\alpha}(G))$.
    \item[(c)] If $G$ is connected and $\mu$ is an eigenvalue of $A_{\alpha}(G)$ with a nonnegative eigenvector, then $\mu = \lambda_1(A_{\alpha}(G))$.
    \item[(d)] If $G$ is connected, and $H$ is a proper subgraph of $G$, then $\lambda_1(A_{\alpha}(H)) < \lambda_1(A_{\alpha}(G))$ for any $\alpha \in [0, 1)$.
\end{enumerate}
\end{proposition}

Proposition \ref{prop::bound_eigenvalues_nikiforov}, Nikiforov introduces a bound for each $A_\alpha-$ eigenvalue based on graph degrees.

\begin{proposition} \label{prop::bound_eigenvalues_nikiforov}
    (\cite{VN17}) Let \( G \) be a graph of order \( n \), with degrees \( \Delta = d_1 \geq d_2 \geq \cdots \geq d_n \). If \( k \in \{1, \cdots, n\} \), then \( \lambda_k(A_\alpha(G)) \leq d_k \). In particular, \( \lambda_1(A_\alpha(G)) \leq \Delta \).

\end{proposition}

The next lemma and proposition, presented by Nikiforov, introduces the spectrum of a regular graph and the $A_\alpha(K_n)$, respectively.

\begin{lemma} \label{lemma::eigeneq_RegularGraphs}
\cite{VN17} If $\alpha \in [0,1]$ and $k = 1, \ldots, n$ and $G$ is a $r$-regular graph of order $n$, then there exists a linear correspondence between the eigenvalues of $A_\alpha(G)$ and $A(G)$, the following way
\begin{equation} \label{eq::autoequation}
\lambda_k(A_\alpha(G)) = \alpha r + (1-\alpha)\lambda_k(A(G)).
\end{equation}
In particular, if $G$ is $r$-regular, then $\lambda_1(A_\alpha(G)) = r, \ \ \forall \alpha \in [0,1]$.
\end{lemma}

\begin{proposition} \label{prop::complete_graph_spectrum}
\cite{VN17} The eigenvalues of $A_\alpha(K_n)$ are $\lambda_1(A_\alpha(K_n)) = n-1$ and $\lambda_k(A_\alpha(K_n)) = \alpha n -1 \text{  for } 2 \leq k \leq n$ and $\alpha \in [0,1]$.
\end{proposition}

The next lemma relates the eigenvalues of the $A_\alpha$-matrix, the number of edges and the first Zagreb index.

\begin{lemma} \label{lemma::sum_eigenvalues}
    (\cite{VN17}) If $G$ is a graph of order $n$ and has $m$ edges, then
    \begin{enumerate}
        \item [(i)] $\displaystyle \sum_{i=1}^n \lambda_i(A_\alpha(G)) = 2m\alpha.$
        \item [(ii)] $\displaystyle \sum_{i=1}^n \lambda_i^2(A_\alpha(G)) = \alpha^2Z_1(G) + (1-\alpha)^22m.$
    \end{enumerate}
\end{lemma}

The following proposition, presented by Lin \textit{et al.} (\cite{LIN2019441}), establishes a criterion for graphs determined by its spectrum.

\begin{proposition}\label{prop::DS_Aalpha}
(\cite{LIN2019441}) Let \( G \) be an \( r-\)regular graph determined by its \( A- \)spectrum (respectively, \( L-\)spectrum or \( Q-\)spectrum). Then \( G \) is also determined by its \( A_\alpha \)-spectrum.
\end{proposition}

\section{Main results} \label{sec::newbounds}

In this section, some new bounds are presented for the largest and smallest eigenvalues of the $A_{\alpha}-$matrix. The first theorem extends the upper bound established by Cioabă (\cite{CIOABA2007483}) for $\lambda_1(A(G))$, addressing the largest eigenvalue of the $A_{\alpha}-$matrix. This upper bound for $\lambda_1(A_\alpha(G))$ uses the order and size, the diameter, the  maximum degree of $G$ and the parameter $\alpha.$

\begin{theorem} \label{prop::cioaba}
    Let $G$ be a connected irregular graph with $n$ vertices, $m$ edges, maximum degree $\Delta$, diameter $diam(G)$ and $\alpha \in [0,1].$ Then
    \begin{equation}
        \lambda_1(A_\alpha(G)) < \Delta - (1-\alpha)\dfrac{n \Delta - 2m}{n(diam(G)(n\Delta - 2m) + 1)}.
    \end{equation}
\end{theorem}
\begin{proof}
    Let $x$ be the unique unit positive eigenvector of $A_\alpha(G)$ associated to the eigenvalue $\lambda_1(A_\alpha(G)).$ From Proposition~\ref{prop:rayleigh_alpha} we have
    \begin{align*}
        \lambda_1(A_\alpha(G)) &= x^TA_\alpha(G)x \\
        &=\alpha x^TD(G)x + (1-\alpha)x^TA(G)x\\
        &=\alpha x^TD(G)x + 2(1-\alpha)\displaystyle \sum_{v_i \sim v_j}x_ix_j\\
        &=\alpha x^TD(G)x + (1-\alpha)\displaystyle\left[\sum_{v_i \sim v_j}(x_i^2 + x_j^2)- \sum_{v_i \sim v_j}(x_i - x_j)^2\right]\\
        &=\alpha x^TD(G)x + (1-\alpha)\displaystyle\left[\sum_{i=1}^n d_ix_i^2- \sum_{v_i \sim v_j}(x_i - x_j)^2\right].
    \end{align*}
    So,
     {\small{
    \begin{equation} \label{eq::eq1_teo3.1}
    (1-\alpha)\Delta - \lambda_1(A_\alpha(G)) = - \alpha x^TD(G)x + (1-\alpha)\displaystyle\left[\sum_{v_i \sim v_j}(x_i - x_j)^2 + \sum_{i=1}^n (\Delta - d_i)x_i^2\right]
    \end{equation}}}
        
    Now, we take vertices $s,t\in \{1, \ldots ,n\}$ such that $\displaystyle x_s= \max_{1 \leq i \leq n}\{x_i\}$ and $\displaystyle x_t=\min_{1 \leq i \leq n}\{x_i\}$. As $G$ is irregular, $t\neq s$. Let $s=i_0,i_1,\ldots, i_{k-1},i_k=t$ are consecutive vertices of a shortest path from $s$ to $t$ in $G$. So, from the Cauchy–Schwarz inequality we have
    $$\sum_{j=0}^{k-1}(x_{i_j} - x_{i_{j+1}})^2 \geq \dfrac{1}{k}\left( \sum_{j=0}^{k-1}(x_{i_j} - x_{i_{j+1}})\right)^2 = \dfrac{1}{k}(x_s - x_t)^2 \geq \dfrac{1}{diam(G)}(x_s - x_t)^2.$$

    So,
    {\small
    \begin{align}
        (1-\alpha)\Delta - \lambda_1(A_\alpha(G)) &\geq - \alpha x^TD(G)x + (1-\alpha)\displaystyle\left[ \dfrac{1}{diam(G)}(x_s - x_t)^2 + (n\Delta - 2m)x_t^2\right] \nonumber\\
        &\geq -\alpha\Delta + (1-\alpha)\displaystyle\left[ \dfrac{1}{diam(G)}(x_s - x_t)^2 + (n\Delta - 2m)x_t^2\right], \label{eq::eq2_teo3.1}
    \end{align}}
    where the right-hand side is a quadratic function of $x_t$ that attains its minimum when
    \begin{equation}\label{eq::eq3_teo3.1}
        x_t = \dfrac{x_s}{diam(G)(n\Delta - 2m)+1}.
    \end{equation}
    From \eqref{eq::eq2_teo3.1} and \eqref{eq::eq3_teo3.1}, we have $$(1-\alpha)\Delta - \lambda_1(A_\alpha(G)) \geq -\alpha\Delta +(1-\alpha)\dfrac{(n\Delta -2m)x_s^2}{diam(G)(n\Delta -2m) + 1}.$$
    As $x^Tx = 1$, we have that $x_s^2 > \dfrac{1}{n}$ and the result follows.
\end{proof}

Based on the upper bound obtained by Stevanovic in (\cite{dragan2004101016}) for $\lambda_1(A(G))$, we obtain an upper bound for $\lambda_1(A_\alpha(G))$ in function of the number of vertices, the maximum degree of the graph and $\alpha$, as we can see in Theorem \ref{theo::theorem2}.

\begin{theorem} \label{theo::theorem2}
     Let $G$ be a connected irregular graph with $n$ vertices, $m$ edges, maximum degree $\Delta$ and $\alpha \in [0,1].$ Then, 
     \begin{equation}
         \lambda_1(A_\alpha(G)) < \Delta - \dfrac{1-\alpha}{2n(n\Delta-1)\Delta^2}.
     \end{equation}
\end{theorem}
\begin{proof}
    Let \(x\) be a positive unit eigenvector of \(A_\alpha(G)\) associated to $\lambda_1(A_\alpha(G))$. From Proposition~\ref{prop:rayleigh_alpha} we have that
    \begin{align*}
        \lambda_1(A_\alpha(G)) &= x^TA_\alpha(G)x \\
        &=\alpha x^TD(G)x + (1-\alpha)x^TA(G)x\\
        &=\alpha \displaystyle \sum_{i=1}^nd_ix_i^2 + 2(1-\alpha)\displaystyle \sum_{v_i \sim v_j}x_ix_j.
    \end{align*}
    Since the maximum degree of \(G\) is \(\Delta\) and \(G\) is irregular, we have that
    \begin{align} \label{theo::ineq1_theo3.2}
        \Delta - \lambda_1(A_\alpha(G)) &> \displaystyle \sum_{i=1}^nd_ix_i^2 - \alpha \displaystyle \sum_{i=1}^nd_ix_i^2 - 2(1-\alpha)\displaystyle \sum_{v_i \sim v_j}x_ix_j \nonumber \\
        &=(1-\alpha)\displaystyle \left(\sum_{i=1}^nd_ix_i^2 - 2\sum_{v_i \sim v_j}x_ix_j \right) \nonumber \\
        &=(1-\alpha)\displaystyle \sum_{v_i \sim v_j}(x_i-x_j)^2. 
    \end{align}
    As $m =   \dfrac{\displaystyle \sum_{i=1}^nd_i}{2} \leq \dfrac{n\Delta - 1}{2},$ from Cauchy–Schwarz inequality we have that
    \begin{equation} \label{eq::ineq2_theo3.2}
        \displaystyle \sum_{v_i \sim v_j}(x_i-x_j)^2 \geq \dfrac{1}{m}\left( \displaystyle \sum_{v_i \sim v_j}\vert x_i-x_j \vert \right)^2 \geq \dfrac{2}{n\Delta - 1}\left( \displaystyle \sum_{v_i \sim v_j}\vert x_i-x_j \vert \right)^2.
    \end{equation}
    Let \(u\) and \(v\) be the vertices of \(G\) such that \(x_u = \displaystyle \max_{1 \leq i \leq n} x_i\) and \(x_v = \displaystyle \min_{1 \leq i \leq n} x_i\) and \(u = w_0, w_1, \ldots, w_k = v\) be a path between \(u\) and \(v\) in \(G\). So
    \begin{equation} \label{eq::ineq3_theo3.2}
        \sum_{v_i \sim v_j} |x_i - x_j| \geq \sum_{l=0}^{k-1} |x_{w_l} - x_{w_{l+1}}| \geq \sum_{l=0}^{k-1} (x_{w_l} - x_{w_{l+1}}) = x_{w_0} - x_{w_k} = x_u - x_v.
    \end{equation}
    From inequalities \eqref{theo::ineq1_theo3.2}, \eqref{eq::ineq2_theo3.2} and \eqref{eq::ineq3_theo3.2} follows that $$\Delta - \lambda_1(A_\alpha(G)) \geq \dfrac{2(1-\alpha)}{n\Delta - 1} (x_u - x_v)^2.$$
    
    \noindent Applying the same techinique used by Stevanovic in \cite{dragan2004101016} we can estimate $x_u - x_v$. Since \( \displaystyle \sum_{i = 1}^n x^2_i = 1 \), we always have that \( x_u > \dfrac{1}{\sqrt{n}} \) and \( x_v < \dfrac{1}{\sqrt{n}} \). Let $c > 0$ be a positive number. So, there are three cases to consider:
    
    \noindent \textit{Case (i)}: \( x_u \geq \dfrac{1}{\sqrt{n}} + c \). Then \( x_v < \dfrac{1}{\sqrt{n}} \) and \( \Delta - \lambda_1(A_\alpha(G)) > \dfrac{2c^2(1-\alpha)}{n\Delta - 1} \).

     \noindent \textit{Case (ii)}: \( x_v \leq \dfrac{1}{\sqrt{n}} - c \). Then \( x_u > \dfrac{1}{\sqrt{n}} \) and again \( \Delta - \lambda_1(A_\alpha(G)) > \dfrac{2c^2(1-\alpha)}{n\Delta - 1} \).

      \noindent \textit{Case (iii)}: \( \dfrac{1}{\sqrt{n}} - c < x_v <  x_u < \dfrac{1}{\sqrt{n}} + c \). Then \( x_i \in \left(\frac{1}{\sqrt{n}} - c, \frac{1}{\sqrt{n}} + c\right) \) $\forall 1 \leq i \leq n$. As $G$ is an irregular graph, there is \( v_s \in V \) such that \( d_s \leq \Delta - 1 \). So, 
      
      \begin{align*}
           \lambda_1(A_\alpha(G)) \left(\dfrac{1}{\sqrt{n}} - c\right) & < \lambda_1(A_\alpha(G))x_s = \displaystyle \alpha d_s x_s + (1-\alpha)\sum_{v_t \sim v_s} x_t\\
           & <  \alpha (\Delta -1)\left(\dfrac{1}{\sqrt{n}} + c\right) + (1-\alpha)(\Delta - 1)\left(\dfrac{1}{\sqrt{n}} + c\right) \\
           & = \left( \Delta - 1 \right) \left(\dfrac{1}{\sqrt{n}} + c\right).
      \end{align*}
      
      Then,
      \[\lambda_1(A_\alpha(G)) < (\Delta -1)\dfrac{1+c\sqrt{n}}{1-c\sqrt{n}}.\]
      Since Nikiforov, in \cite{VN17}, proved that $\lambda_1(A_\alpha(G)) \leq \Delta$ and in order for the expression on the right-hand side to be useful, it must be less than \( \Delta \), then it is satisfied for $c < \dfrac{1}{(2\Delta-1)\sqrt{n}}.$ Making \( c = \dfrac{1}{2\Delta\sqrt{n}} \), in cases (i) and (ii) we get that \( \lambda_1(A_\alpha(G)) \leq \Delta - \dfrac{1 - \alpha}{2n(n\Delta - 1)\Delta^2} \), while in case (iii) we get that \( \lambda_1(A_\alpha(G)) \leq \Delta - \dfrac{1}{2\Delta - 1} < \Delta -\dfrac{1 - \alpha}{2n(n\Delta - 1)\Delta^2}\).
\end{proof}

Let \( H_{n-1, \Delta_2} \) denote a connected graph with maximum degree \( n-1 \) and a second maximum degree \( \Delta_2 \), where \( \Delta_2 = \delta(G) \) and \( \Delta_2 < n-1 \). Consequently, \( H_{n-1, \Delta_2} \) is a semiregular graph. From Theorem~\ref{theo::equitable_eigenvalues_relation} can easily see that the largest eigenvalue of \( A_{\alpha}(H_{n-1, \Delta_2}) \) is given by the largest root of the following equation:

\[
\lambda^2 + (-\alpha n - \Delta_2 + 1)\lambda + (n-1)(\alpha \Delta_2 + \alpha -1) = 0.
\]

So, $\lambda_1(A_{\alpha}(H_{n-1, \Delta_2}))$ is given by

\begin{equation}\label{eq::largest_H}
    \dfrac{\alpha n + \Delta_2 - 1 +\sqrt{\alpha^2 n^2 + (\Delta_2 - 1)^2 -2\alpha((n-2)\Delta_2 + 3n-2) + 4(n-1)}}{2}.
\end{equation}

The next theorem presents an upper bound for $\lambda_1(A_\alpha(G))$ based on the upper bound for  $\lambda_1(A(G))$, presented by K. Das in \cite{DAS20112420}. 

\begin{theorem}
    Let $G$ be a connected graph with maximum degree $\Delta$, second maximum degree $\Delta_2$ and $\alpha \in [0,1].$ Then,
    {\footnotesize{
    \begin{equation}\label{eq::theo3.3}
      \lambda_1(A_\alpha(G)) \leq \dfrac{\alpha(\Delta + 1)+ \Delta_2 - 1 + \sqrt{\alpha^2(\Delta + 1)^2 + (\Delta_2 - 1)^2 -2\alpha((\Delta-1)\Delta_2 + 3\Delta + 1) + 4\Delta}}{2}
    \end{equation}}}
    The equality in the above inequality holds if and only if $G$ is a $\Delta-$regular graph or $G \cong H_{n-1,\Delta_2}.$
\end{theorem}
\begin{proof}
 From Theorem~\ref{theo::bounds_adjacency}, Proposition~\ref{prop::bound_eigenvalues_nikiforov} and Proposition~\ref{prop::DS_Aalpha}, it is well-known that $\lambda_1(A_\alpha(G)) \leq \Delta$, with equality if and only if $G$ is $\Delta-$regular. If \(\Delta = \Delta_2\), then
    \begin{equation*}
        \lambda_1(A_\alpha(G)) \leq \Delta = \dfrac{\alpha(\Delta+1) + \Delta -1 + \sqrt{(\Delta+1)^2(1-\alpha)^2}}{2}
    \end{equation*}
    and the result follows. 
    
    Now, consider $\Delta \neq \Delta_2.$ Let $x = (x_1, x_2, \ldots, x_n)^T$ be an eigenvector of $A_\alpha(G)$ corresponding to the eigenvalue $\lambda_1(A_\alpha(G))$. Without loss of generality, we can assume that exist $i$ such that $x_i = 1$ and the other components are less than or equal to 1. Let $\displaystyle x_j = \max_{\substack{1 \leq k \leq n \\ k \neq i}} x_k.$ We need to consider two cases:

    \noindent \textit{Case (i):} $d_j = \Delta$ and $d_i \leq \Delta_2.$ So,
    \begin{equation*}
        \displaystyle \lambda_1(A_\alpha(G))x_i = \alpha d_ix_i + (1-\alpha)\sum_{v_i \sim v_k}x_k.
    \end{equation*}
    As $x_i = 1$ and $x_j < 1,$ we have 
    \begin{align*}
        \displaystyle \lambda_1(A_\alpha(G)) & = \alpha d_i + (1-\alpha)\sum_{v_i \sim v_k}x_k\\
        & \leq \alpha d_i + (1-\alpha) d_i x_j\\
        & \leq \alpha d_i + (1-\alpha) d_i\\
        & \leq \Delta_2.
    \end{align*}
    Moreover, we can see that
    \begin{align*}
        \Delta_2 & \leq \dfrac{\Delta_2 - 1 + \sqrt{(\Delta_2 - 1)^2 + 4\Delta}}{2}\\
        & \leq \dfrac{\alpha(\Delta + 1)+ \Delta_2 - 1 + \sqrt{\alpha^2(\Delta + 1)^2 + (\Delta_2 - 1)^2 -2\alpha((\Delta-1)\Delta_2 + 3\Delta + 1) + 4\Delta}}{2}
    \end{align*}
    with equality if  and only if $\Delta = \Delta_2.$ Since $\Delta \neq \Delta_2,$ we get
    {\footnotesize{
    \begin{equation*}
        \lambda_1(A_\alpha(G)) \leq \Delta_2 < \dfrac{\alpha(\Delta + 1)+ \Delta_2 - 1 + \sqrt{\alpha^2(\Delta + 1)^2 + (\Delta_2 - 1)^2 -2\alpha((\Delta-1)\Delta_2 + 3\Delta + 1) + 4\Delta}}{2}.
    \end{equation*}}}
    
    \noindent \textit{Case (ii):} $d_j \leq \Delta_2.$ So,
    \begin{equation*}
        \displaystyle \lambda_1(A_\alpha(G))x_i = \alpha d_i x_i + (1-\alpha)\sum_{v_i \sim v_k}x_k.
    \end{equation*}
    As $x_i = 1$ and $d_i \leq \Delta$, follows
    \begin{equation}\label{eq::eq1_theo3.3}
        \lambda_1(A_\alpha(G)) \leq \alpha d_i + (1-\alpha)d_ix_j \leq \alpha \Delta + (1-\alpha)\Delta x_j. 
    \end{equation}
    Moreover, we have
    \begin{align*}
        \displaystyle \lambda_1(A_\alpha(G))x_j & = \alpha d_j x_j + (1-\alpha)\sum_{v_j \sim v_k} x_k\\
        & \leq \alpha d_j x_j + (1-\alpha)(1 + (d_j-1)x_j)\\
        & \leq \alpha \Delta_2 x_j + (1-\alpha)(1 + (\Delta_2-1)x_j).
    \end{align*}
    So,
    \begin{equation*}
        (\lambda_1(A_\alpha(G)) -\alpha \Delta_2 - (1-\alpha)(\Delta_2-1))x_j \leq 1-\alpha,
    \end{equation*}
   and consequently
    \begin{equation} \label{eq::eq2_theo3.3}
        (\lambda_1(A_\alpha(G)) -\Delta_2  + 1-\alpha)x_j \leq 1-\alpha.
    \end{equation}
    
    From \eqref{eq::eq1_theo3.3} and \eqref{eq::eq2_theo3.3} follows that
    \begin{align*}
        \lambda_1(A_\alpha(G))(\lambda_1(A_\alpha(G)) -\Delta_2 + 1-\alpha) &\leq \alpha\Delta(\lambda_1(A_\alpha(G)) -\Delta_2 + 1-\alpha) + \\ 
        & \qquad (1-\alpha)\Delta(\lambda_1(A_\alpha(G)) -\Delta_2 + 1-\alpha)x_j \\
        & \leq \alpha\Delta(\lambda_1(A_\alpha(G)) -\Delta_2 + 1-\alpha) + \\
        &\qquad (1-\alpha)^2\Delta.
    \end{align*}
    So, 
    \begin{equation*}
        \lambda_1^2(A_\alpha(G)) + (-\Delta_2 + 1- \alpha -\alpha\Delta)\lambda_1(A_\alpha(G)) + \alpha\Delta \Delta_2 + (\alpha-1)\Delta \leq 0.
    \end{equation*}
    Solving this inequality, the result follows.

    If $G$ is $\Delta-$regular graph, from Lemma~\ref{lemma::eigeneq_RegularGraphs} we have that $\lambda_1(A_\alpha(G))=\Delta$ and making algebraic manipulations we conclude the equality occurrence. If $G \cong H_{n-1,\Delta_2}$ we have that the largest eigenvalue of $G$ is given by \eqref{eq::largest_H}. So, replacing the $G$ details and making some algebraic manipulation we conclude the equality.
    
    Now, suppose that equality holds. If $\Delta_2 = \Delta$, then $G$ is isomorphic to a regular graph. Otherwise, if $\Delta \neq \Delta_2$, all the inequalities in the above argument in Case (ii) must be equalities. From inequality \eqref{eq::eq1_theo3.3}, we obtain that $x_k = x_j$ for all $v_k \in N(v_i)$ and $d_i = \Delta$. Additionally, from \eqref{eq::eq2_theo3.3}, we get $x_k = x_j$ for all $v_k \in N(v_j)$, with $k \neq i$ and $d_j = \Delta_2$. Since $d_i = \Delta$, it follows that $v_iv_j \in E(G)$. Let $V_1 = \{v_k: x_k = x_j\}$. If $V_1 \neq V(G) \backslash \{v_i\}$ and as $G$ is connected, then there exist vertices $v_p \in V_1$ and $v_q \notin V_1$, with $q \neq i$, such that $v_pv_q \in E(G)$ and consequently $x_q < x_j.$  In this case, we have $\lambda_1(A_\alpha(G))x_q < \lambda_1(A_\alpha(G))x_p = \lambda_1(A_\alpha(G))x_j < \alpha \Delta_2 x_j + (1-\alpha)(1 + (\Delta_2-1)x_j)$, which leads to a contradiction. Therefore, $V_1 = V(G) \backslash \{v_i\}$. If $x_j = 1$, then $\lambda_1(A_\alpha(G)) = \Delta = \Delta_2$, which is a contradiction since $\Delta \neq \Delta_2.$ Then $x_j < 1$. Suppose that exist a vertex \(v_k \in V_1\) that is not adjacent to the vertex \(v_i\). Since \(d_k \leq \Delta_2\) and \(x_j < 1,\) we have that $\lambda_1(A_\alpha(G))x_k = \lambda_1(A_\alpha(G))x_j < \alpha \Delta_2 x_j + (1-\alpha)(1 + (\Delta_2-1)x_j)$ which is a contradiction. Thus, we conclude that each vertex \(v_k \in V_1\) is adjacent to vertex \(v_i\) and \(d_k = \Delta_2 = \delta < n - 1\). Hence, $G \cong H_{n-1, \Delta_2}$ concluding the result.
\end{proof}

In the next theorem, we present an upper bound for $\lambda_1(A_\alpha(G))$ whose proof uses the first Zagreb index, the sum of the eigenvalues and the sum of the squares of the eigenvalues associated to $A_\alpha(G)$.

\begin{theorem} \label{theo::largest_upper_with_zagrab}
    Let $G$ be a graph with $n$ vertices, $m$ edges, $Z_1(G) = Z_1$ and $\alpha \in [0,1]$. Then, 
    \begin{equation}\label{eq::upperbound4}
        \lambda_1(A_\alpha(G)) \leq \dfrac{2 \alpha m + \sqrt{(n - 1) \left( \alpha^2 \left( -4 m^2 + 2 m n + n Z_1 \right) - 4 \alpha m n + 2 m n \right)}}{n}.
    \end{equation}
\end{theorem}
If $G \simeq K_n$, then the equality occurs.
\begin{proof}
    From Lemma~\ref{lemma::sum_eigenvalues}, item \textit{(i),} we have
    \begin{align*}
        2m\alpha - \lambda_1(A_\alpha(G)) & = \displaystyle  \sum_{i=2}^n \lambda_i(A_\alpha(G)) \\
        & \leq \displaystyle  \sum_{i=2}^n \vert\lambda_i(A_\alpha(G))\vert.
    \end{align*}
    
    From Lemma~\ref{lemma::sum_eigenvalues}, item \textit{(ii),} and the Cauchy–Schwarz inequality we obtain
    \begin{align*}
        \alpha^2Z_1 + (1-\alpha)^22m - \lambda_1^2(A_\alpha(G)) & = \displaystyle \sum_{i=2}^n\lambda_i^2(A_\alpha(G)) \\
        & \geq \dfrac{1}{n-1} \left( \displaystyle \sum_{i=2}^n \vert \lambda_i(A_\alpha(G)) \vert \right)^2\\
        & \geq \dfrac{(2m\alpha - \lambda_1(A_\alpha(G)))^2}{n-1}
    \end{align*}
    Hence,
    $$\alpha^2Z_1 + (1-\alpha)^22m - \lambda_1^2(A_\alpha(G)) - \dfrac{(2m\alpha - \lambda_1(A_\alpha(G)))^2}{n-1} \geq 0.$$
    and therefore,
    $$\lambda_1(A_\alpha(G)) \leq\dfrac{2 \alpha m + \sqrt{(n - 1) \left( \alpha^2 \left( -4 m^2 + 2 m n + n Z_1 \right) - 4 \alpha m n + 2 m n \right)}}{n}.$$

    If $G \simeq K_n$, then $m = \dfrac{n(n-1)}{2}$, $Z_1(G) = n(n-1)^2$ and from Proposition \ref{prop::complete_graph_spectrum} $\lambda_1(A_\alpha(K_n)) = n-1$, so the equality occurs. 
\end{proof}

\begin{corollary}
    Let $G$ be a graph with $n$ vertices, $m$ edges and $\alpha \in [0,1]$. Then,
    \small{
    \begin{equation}
        \lambda_1(A_\alpha(G)) \leq \dfrac{2 \alpha m + \sqrt{(n - 1) \left( - 4 \alpha^2 m^2 + 2 m n\left((n+\delta)\alpha^2 - 2\alpha +1 \right) -n^2\alpha^2(n-1)\delta\right)}}{n}
    \end{equation}}
\end{corollary}
\begin{proof}
    From  Theorem~\ref{theo::largest_upper_with_zagrab} and Theorem~\ref{theo::sharp_DAS_upper} the result follows.
\end{proof}

Theorem~\ref{theo::alpha-liu2008} introduces an upper bound for $\lambda_1(A_\alpha(G))$ based on the bound for $\lambda_1(Q(G)),$ presented by Liu in \cite{Liu2008}. 

\begin{theorem}\label{theo::alpha-liu2008}
    Let $G$ be a graph on $n$ vertices with maximum degree $\Delta$, clique number $\omega$ and $\alpha \in [0,1].$ Then,
    \begin{equation}
       \lambda_1(A_\alpha(G)) \leq \alpha \Delta + (1-\alpha)\left(1 - \dfrac{1}{\omega}\right)n. 
    \end{equation}
\end{theorem}
\begin{proof}
    Let $y = (y_1, y_2, \ldots, y_n)$ be the normalized eigenvector corresponding to $\lambda_1(A_\alpha(G)).$ Let $\displaystyle U = \sum_{i = 1}^ny_i$ and $x = \dfrac{1}{U}\cdot y.$ Then,
    \begin{equation*}
        \langle x, A_\alpha(G)x \rangle = \langle x, (\alpha D(G) + (1-\alpha)A(G))x \rangle = \alpha \langle x, D(G)x \rangle + (1-\alpha)\langle x, A(G)x \rangle.\\
    \end{equation*}
    
    From Theorem~\ref{theo::motzkin} we have that
    \begin{equation*}
         \langle x, A_\alpha(G)x \rangle \leq \alpha \Delta \langle x, x \rangle + (1-\alpha) \left( 1 - \dfrac{1}{\omega} \right) = \dfrac{\alpha \Delta}{U^2} + (1-\alpha)\left(1 - \dfrac{1}{\omega} \right).
    \end{equation*}

    On the other hand,
    \begin{align*}
        \langle x, A_\alpha(G)x \rangle &= \displaystyle (2\alpha -1)\sum_{i=1}^n x_i^2 d_i + (1-\alpha)\sum_{v_i \sim v_j}(x_i + x_j)^2\\
        & = \dfrac{1}{U^2}\left(\displaystyle (2\alpha -1)\sum_{i=1}^n y_i^2 d_i + (1-\alpha)\sum_{v_i \sim v_j}(y_i + y_j)^2 \right)\\
        & = \dfrac{\lambda_1(A_\alpha(G))}{U^2}.
     \end{align*}

     Therefore
     \begin{equation*}
         \dfrac{\lambda_1(A_\alpha(G))}{U^2} \leq \dfrac{\alpha \Delta}{U^2} + (1-\alpha)\left(1 - \dfrac{1}{\omega} \right)
     \end{equation*}
    and so,
    \begin{equation*}
        \lambda_1(A_\alpha(G)) \leq \alpha\Delta + U^2(1-\alpha)\dfrac{\omega - 1}{\omega}.
    \end{equation*}

     Since $U^2 = \displaystyle \left( \sum_{i = 1}^n y_i \right)^2 \leq n\sum_{i=1}^n y_i^2 = n,$ we obtain
     \begin{equation*}
         \lambda_1(A_\alpha(G)) \leq \dfrac{\alpha \Delta \omega + (1-\alpha)(\omega - 1)n}{\omega}
     \end{equation*}
     and the result follows.
\end{proof}

Theorem \ref{theo::cioaba} introduces a lower bound for $\lambda_n(A_\alpha(G))$ based on a bound obtained by Alon and Sudakov, in \cite{alon2000}, for the smallest eigenvalue of the adjacency matrix of $G$.

\begin{theorem}\label{theo::cioaba}
    Let $G$ be a graph on $n$ vertices with diameter $diam(G)$, maximum degree $\Delta$ and $\alpha \in [0,1]$. If G is non-bipartite then
    \begin{equation}
        \lambda_n(A_\alpha(G)) \geq -\Delta + \dfrac{1+\alpha}{n(diam(G) + 1)}
    \end{equation}
\end{theorem}
\begin{proof}
    Let $G$ be a graph, $A_\alpha(G) = A_\alpha$ and $V(G) = \{ v_1, \ldots, v_n \}.$ Let $x = (x_1, \ldots, x_n)$ be an unit eigenvector corresponding to the smallest eigenvalue $\lambda_n(A_\alpha)$. Then,
    \begin{align*}
        \lambda_n(A_\alpha) & = x^TA_\alpha x = \alpha x^T D(G) x + (1-\alpha)x^T A(G) x\\
        %& = \alpha \displaystyle \sum_i d_i x_i^2 + (1-\alpha) \sum_{ij}a_{ij}x_ix_j\\
        & = \alpha \displaystyle \sum_{i=1}^{n} d_i x_i^2 + 2(1-\alpha) \sum_{v_i \sim v_j}x_ix_j.
    \end{align*}

    %The fact that the maximum degree of \( G \) is \( \Delta \), together with the inequality $\Delta = \Delta \vert x \vert^2 \geq \sum_{i} d_i x_i^2$ implies that
    As $\Delta = \Delta \vert x  \vert^2 \geq \displaystyle \sum_{i=1}^n d_i x_i^2,$ we have that
    \begin{align}\label{theo3.7}
        \Delta + \lambda_n(A_\alpha) & \geq \displaystyle \sum_{i=1}^nd_ix_i^2 + \alpha \sum_{i=1}^nd_ix_i^2 + 2(1-\alpha) \sum_{v_i \sim v_j}x_ix_j \nonumber\\
        & = \left( \displaystyle \sum_{i=1}^nd_ix_i^2 +  2\sum_{v_i \sim v_j}x_ix_j\right) + \alpha \left( \displaystyle \sum_{i=1}^nd_ix_i^2 - 2\sum_{v_i \sim v_j}x_ix_j\right) \nonumber\\
        & = \displaystyle \sum_{v_i \sim v_j}(x_i + x_j)^2 + \alpha \sum_{v_i \sim v_j}(x_i - x_j)^2.
    \end{align}

    The set $V(G)$ can be partitioned into two subsets, $V_1$ and $V_2$, such that if $v_s \in V_1$ if and only if $x_s$ is negative number and $V_2$ be the set of all remaining vertices. As $G$ is non-bipartite there exists an edge $v_sv_t$ of $G$ such that either both coordinates $x_s$ and $x_t$ are non-negative or both are negative.
    
    Without loss of generality, we can suppose that \( x_1 \) is non-negative and has maximum value among all entries of \( x \). First, consider the case that \( v_sv_t \in E \) and \( x_s \geq 0, x_t \geq 0 \). Now, we can assume that \( d(v_1, v_k) + d(v_s,v_t) \leq diam(G) + 1 \) such that \( k \in \{s, t\} \). Let \( v_1 = w_1, w_2, \ldots, w_k = v_k \) be this such path. From inequality (\ref{theo3.7})  and the Cauchy-Schwarz inequality, we have

   \begin{align*}
    \Delta + \lambda_n(A_\alpha) & \geq \sum_{v_i \sim v_j}(x_i + x_j)^2 + \alpha \sum_{v_i \sim v_j}(x_i - x_j)^2 \\
    & \geq \sum_{i=1}^{k-1}(x_{w_i} + x_{w_{i+1}})^2 + \alpha \sum_{i=1}^{k-1}(x_{w_{i}} - x_{w_{i+1}})^2 \\
    & \geq \dfrac{1}{k-1} \left[ \left(\sum_{i=1}^{k-1}\vert x_{w_{i}} + x_{w_{i+1}} \vert \right)^2 + \alpha\left( \sum_{i=1}^{k-1}(x_{w_{i}} - x_{w_{i+1}})\right)^2 \right] \\
    & \geq \dfrac{1}{k-1} \left[ \left((x_{w_{1}} + x_{w_{2}}) + (-x_{w_{2}} - x_{w_{3}}) + (x_{w_{3}} + x_{w_{4}}) + (-x_{w_{4}} - x_{w_{5}})\right.\right. \\
    & \left. \left. + \cdots + (x_{k-1} + x_k)\right)^2 + \alpha \left((x_{w_{1}} - x_{w_{2}}) + (x_{w_{2}} - x_{w_{3}}) + (x_{w_{3}} - x_{w_{4}}) \right.\right.\\
    & \left.\left. + (x_{w_{4}} - x_{w_{5}}) + \cdots + (x_{w_{k-1}} - x_{w_{k}})\right)^2 \right]\\
    & = \dfrac{1}{k-1} \left[(x_{w_{1}} + x_{w_{k}})^2 + \alpha(x_{w_{1}} - x_{w_{k}})^2 \right]\\
    & \geq \dfrac{1}{diam(G) + 1}\left( x_1^2 + \alpha x_1^2 \right) = \dfrac{1}{diam(G) + 1}(1+\alpha)x_1^2.
\end{align*}

Finally, since $\displaystyle \sum_{i = 1}^nx_i^2  = 1$ and $x_1$ has maximum value, we conclude that
$$\Delta + \lambda_n(A_\alpha) \geq \dfrac{x_1^2(1+\alpha)}{diam(G) + 1} \geq \dfrac{ \displaystyle \sum_{i=1}^nx_i^2}{n(diam(G) + 1)}(1+\alpha) = \dfrac{1+\alpha}{n(diam(G) + 1)}.$$

Now consider the case that both coordinates \( x_s, x_t \), such that \(v_sv_t \in E \), are negative. Using the reasoning above it follows that there exists a path from \( v_1 \) to \(v_k \), $k \in \{s,t\}$ such that \( x_k < 0 \) and \( d(v_1, v_k) + d(v_s,v_t) \leq diam(G) + 1 \). Thus in a similar way, we get the result.
\end{proof}

As consequence of Theorem \ref{theo::cioaba} and Lemma \ref{lemma::eigeneq_RegularGraphs}, we obtain following result for $\Delta-$regular graphs.

\begin{corollary}
    Let $G$ be a $\Delta-$regular graph on $n$ vertices. If $G$ is non-bipartite graph then
    $$\lambda_1(A_\alpha(G)) + \lambda_n(A_\alpha(G)) \geq \dfrac{1+\alpha}{n(diam(G) + 1)}.$$
\end{corollary}

Theorem~\ref{theo::alpha-liu2008-smallest} introduces an upper and a lower bound for $\lambda_n(A_\alpha(G))$ based on the bound for $\lambda_n(Q(G))$, presented by Liu in \cite{Liu2008}. 

\begin{theorem}\label{theo::alpha-liu2008-smallest}
    Let $G$ be a graph with $n$ vertices, $m$ edges, $\delta$ and $\omega$ its minimum degree and clique number respectively and $\alpha \in [0,1]$. Then,
    \begin{equation}
        \alpha \delta - \sqrt{2m(1-\alpha)\left(1-\dfrac{1}{\omega}\right)} \leq \lambda_n(A_\alpha(G)) \leq \alpha \delta + \sqrt{2m(1-\alpha)\left(1-\dfrac{1}{\omega}\right)}
    \end{equation}
\end{theorem}
\begin{proof}
    Let $y = (y_1, \ldots, y_n)$ be the normalized eigenvector corresponding to $\lambda_n(A_\alpha(G))$. Then
    \begin{align} \label{ineq::ineq1_theo3.8}
        \lambda_n(A_\alpha(G)) &= y^TA_\alpha(G)y = y^T(\alpha D(G) + (1-\alpha)A(G))y \nonumber\\
        & = \alpha y^TD(G)y + (1-\alpha)y^TA(G)y\nonumber \\
        & = \displaystyle \alpha \sum_{i = 1}^n y_i^2 d_i + (1-\alpha)\sum_{v_i \sim v_j}2y_iy_j \nonumber\\
        & \displaystyle \geq \alpha \delta + (1-\alpha)\sum_{v_i \sim v_j} 2y_iy_j
    \end{align}
    Nikiforov proved in \cite{VN17} that $\lambda_n(A_\alpha(G)) \leq \alpha \delta$ and consequently $\lambda_n(A_\alpha(G)) - \alpha \delta \leq 0.$ So, from inequality \eqref{ineq::ineq1_theo3.8} follows that $(\lambda_n(A_\alpha(G)) - \alpha\delta)^2 \leq \displaystyle \left((1-\alpha)\sum_{v_i \sim v_j}2y_iy_j\right)^2.$

    By the Cauchy inequality we obtain
    \begin{equation}
        \displaystyle \left(\sum_{v_i \sim v_j}2y_iy_j\right)^2 \leq 4m \sum_{v_i \sim v_j}(y_iy_j)^2 = 2m\left(2\sum_{v_i \sim v_j}y_i^2y_j^2 \right)
    \end{equation}

    Since $(y_1^2, \ldots, y_n^2)$ is a non-negative vector and $\displaystyle \sum_{i=1}^ny_i^2 = 1$, from Theorem~\ref{theo::motzkin} we have
    $$\displaystyle 2\sum_{v_i \sim v_j}y_i^2y_j^2 \leq 1 - \dfrac{1}{\omega}.$$

    Therefore
    $$\dfrac{(\lambda_n(A_\alpha(G)) - \alpha \delta)^2}{2m} \leq (1-\alpha)\left(1 - \dfrac{1}{\omega}\right).$$
    So,
    $$\alpha \delta - \sqrt{2m(1-\alpha)\left(1-\dfrac{1}{\omega}\right)} \leq \lambda_n(A_\alpha(G)) \leq \alpha \delta + \sqrt{2m(1-\alpha)\left(1-\dfrac{1}{\omega}\right)}.$$
\end{proof}

\section*{Acknowledgments}

\vspace{0.2cm}

\noindent \textbf{Funding information}: The authors thank FAPERJ (E-26/210.109/2023(284573)) for partially funding this research.

\bibliographystyle{unsrt}  
\bibliography{references}  %%% Remove comment to use the external .bib file (using bibtex).

\begin{comment}

\end{comment}

\end{document}